\documentclass[12pt]{amsart}
\usepackage{amssymb}
\usepackage{amsfonts}
\usepackage{amsmath,latexsym,amssymb,amsfonts,amsbsy, amsthm}
\usepackage[usenames]{color}
\usepackage{xspace,colortbl}
\usepackage{mathrsfs}
\usepackage[colorlinks,linkcolor=blue,citecolor=blue]{hyperref}
\usepackage[titletoc]{appendix}
\allowdisplaybreaks

\newcommand{\beq}{\begin{equation}}
\newcommand{\eeq}{\end{equation}}
\newcommand{\ben}{\begin{eqnarray}}
\newcommand{\een}{\end{eqnarray}}
\newcommand{\beno}{\begin{eqnarray*}}
\newcommand{\eeno}{\end{eqnarray*}}
\newcommand{\R}{\mathbb{R}}
\newtheorem{thm}{Theorem}[section]
\newtheorem{defi}[thm]{Definition}
\newtheorem{lem}[thm]{Lemma}
\newtheorem{prop}[thm]{Proposition}
\newtheorem{coro}[thm]{Corollary}
\newtheorem{rmk}[thm]{Remark}

\begin{document}

\title[Stefan as vanishing viscosity limit]{A generalized one phase Stefan problem as a vanishing viscosity limit$^*$}
\author[K. Wang]{Kelei Wang$^\dag$}
\thanks{$^\dag$School of Mathematics and Statistics  \&
 Hubei Key Laboratory of Computational Science, Wuhan University, Wuhan, 430072 China.
{Email: wangkelei@whu.edu.cn}. }

\thanks{$*$  Research  supported
by NSF of China grants no. 11871381 and no. 11631011.}
\date{\today}

\begin{abstract}
We study the vanishing viscosity limit of a nonlinear diffusion equation
describing chemical reaction interface or the spatial segregation interface of competing species, where the diffusion rate for the negative part of the solution converges to zero. As in the standard one phase Stefan problem, we prove that the positive part of the solution converges uniformly to the solution of a generalized one phase Stefan problem. This information is  then employed to determine the limiting equation for the negative part, which is an ordinary differential equation.
\end{abstract}


\maketitle
\renewcommand{\theequation}{\thesection.\arabic{equation}}
\setcounter{equation}{0}


\section{Introduction}\label{sec introduction}
\setcounter{equation}{0}

In this paper we study the convergence  as $\varepsilon\to0$  of solutions to the following nonlinear parabolic problem:
\begin{equation}\label{eqn}
  \partial_tu_\varepsilon=\Delta \alpha_\varepsilon(u_\varepsilon)+f(u_\varepsilon). 
\end{equation}
The function $\alpha_\varepsilon$ is
\begin{equation*}
\alpha_\varepsilon(u):=
  \begin{cases}
     u, & \mbox{if } u\geq 0 \\
    \varepsilon u, &  \mbox{if } u\leq 0.
  \end{cases}
\end{equation*}
The nonlinearity $f$ is assumed to be Lipschitz and satisfies $f(0)=0$.

This equation arises in   chemical reaction models (see Cannon and Hill \cite{Cannon-Hill}, Evans \cite{Evans1}) and the spatial segregated limit of competing systems for two species in ecology models, that is, the $k\to+\infty$ limit of the following system (here we omit intra-species terms):
\begin{equation}\label{competing system}
   \left\{
\begin{aligned}
& \partial_tu_1-d_1\Delta u_1=-ku_1u_2, \\
& \partial_tu_2-d_2\Delta u_2=-ku_1u_2,
\end{aligned}\right.
\end{equation}
see Evans \cite{Evans2}, Dancer et. al. \cite{Hilhorst2} and Crooks et. al. \cite{Dancer}. See  also \cite{Wang} for the case with more than two species.

For \eqref{eqn}, $u^+:=\max\{u,0\}$ and $u^-:=\max\{-u,0\}$ represent the density of two species, while the nodal set $\{u=0\}$ is the segregated interface between them. Regularity properties of the solution and of the interface have been studied by  Cannon and Hill in \cite{Cannon-Hill} and Tonegawa in \cite{Tonegawa}.

In this paper we are interested in the situation where the diffusion rate of one species is so small that negligible. For example, in the liquid-solid phase transition model, sometimes we make the ideal assumption that there is only heat diffusion inside the liquid phase and exchange of heat across the phase interface, but no heat diffusion inside the solid phase. For \eqref{eqn} this corresponds to letting $\varepsilon\to0$ in \eqref{eqn}, hence a kind of vanishing viscosity problem. We will show   the limit is a generalized one phase Stefan free boundary problem, whose weak formulation (in the distributional sense) is as follows (cf. Rodrigues \cite[Section 1.2]{Rodrigues} or Visintin \cite[Section 1.1]{Visintin}).
\begin{defi}
Given a domain $\Omega\subset\R^n$, a constant $T>0$, a nonnegative function $W\in L^\infty(\Omega)$ and a Lipschitz function $f$ on $\R$, a nonnegative function $u\in L^1(\Omega\times(0,T))$ is a weak solution of the generalized, nonlinear one phase Stefan free boundary problem associated to $f$ and $W$, if for any $\eta\in C_0^\infty(\Omega\times(0,T))$, we have
\begin{equation}\label{limiting eqn}
  \int_{0}^{T}\int_{\Omega} \left[\beta(u)\eta_t+ u \Delta \eta+f(u)\eta\right]dxdt=0.
\end{equation}
Here
 \[
   \beta(u)(x,t):=
  \begin{cases}
     u(x,t), & \mbox{if } u(x,t)> 0 \\
     -W(x), &  \mbox{otherwise}.
  \end{cases}
  \]
\end{defi}
If $W\equiv const.$ and $f\equiv 0$, this is the standard one phase Stefan problem. But for the limiting problem of \eqref{eqn}, $W$ is determined by an ordinary differential equation (the limiting  equation for $u_\varepsilon^-$, see \eqref{limit eqn in negative} for the precise statement) and is not a constant in general.

The nonlinear Stefan problem describing biological spreading has   been studied by Du and his collaborators in a series of works, see  \cite{Du-Guo} and the survey paper \cite{Du}, as well as \cite{Du2} where the Stefan free boundary condition is derived based on an ecology consideration. The derivation of a similar one phase Stefan problem from a system similar to \eqref{competing system} has been conducted by Hilhorst et. al. in a series of works \cite{Hilhorst3, Hilhorst4, Hilhorst5}, where they considered the spatial segregation limit $k\to+\infty$ of the system
\begin{equation}\label{competing system 2}
   \left\{
\begin{aligned}
& \partial_tu_1-d_1\Delta u_1=-ku_1u_2, \\
& \partial_tu_2=-ku_1u_2,
\end{aligned}\right.
\end{equation}
In the same spirit, a  two phase Stefan free boundary problem is derived from a modified system of \eqref{competing system}   by Hilhorst et. al. in \cite{Hilhorst,Hilhorst1}.

 Our result, combined with the ones from \cite{Hilhorst2} and  \cite{Wang}, gives a derivation of this Stefan problem in two steps: first \eqref{eqn} is derived from \eqref{competing system} as a spatial segregation limit, then the Stefan problem is derived from \eqref{eqn} as a vanishing viscosity limit. Note that \eqref{competing system 2} can also be viewed as a vanishing viscosity limit of \eqref{competing system}.
Therefore the derivation in \cite{Hilhorst3, Hilhorst4, Hilhorst5} differs from ours in the order of these two steps.

To conclude this introductory section, we give a formal derivation of \eqref{limiting eqn} from \eqref{eqn}, in the setting of classical solutions. (See Visintin \cite[Section 1.2]{Visintin} for the definition of classical solutions of \eqref{limiting eqn}. For \eqref{eqn}, a solution $u_\varepsilon$ is classical  if the interface $\{u_\varepsilon=0\}$ is a smooth hypersurface, see Tonegawa \cite{Tonegawa}.)

  Take a point $(x_0,t_0)\in\partial\{u>0\}$ and assume in a neighborhood of it the free boundary is a  smooth hypersurface, given by the graph $\{t=T(x)\}$. 
Take a hyperbolic scaling
\[\widetilde{u}_\varepsilon(x,t):=u_\varepsilon(x_0+\varepsilon x, t_0+\varepsilon t).\]
Then
\[\widetilde{u}_\varepsilon(x,t)\to \widetilde{u}(x,t):= A\left[1-e^{\frac{t+\xi\cdot x}{|\xi|^2}}\right] \quad \mbox{locally smoothly in } \{t<\xi\cdot x\}.\]
Here $\xi:=\nabla T(x_0)$ and
\[A:=  \lim_{t \to T(x_0), ~~ t<T(x_0)}\lim_{\varepsilon\to 0}u_\varepsilon(x_0,t).\]
In fact, $\widetilde{u}$ is understood as the solution of
\[
  \left\{
\begin{aligned}
& \partial_t\widetilde{u}-\Delta\widetilde{u}=0   \quad & \mbox{in } \{t<\xi\cdot x\},\\
& \widetilde{u}=0   \quad  &\mbox{on } \{t= \xi\cdot x\}.
\end{aligned}\right.
  \]
Note that we have
\[|\nabla \widetilde{u}|^2=A^2|\xi|^{-2}=-A\partial_t\widetilde{u}.\]
Combining this equation with the relations
\[|\nabla u_\varepsilon^+|^2=\varepsilon^2|\nabla u_\varepsilon^-|^2 \quad \mbox{and} \quad \partial_tu_\varepsilon^+=\partial_tu_\varepsilon^- \quad \mbox{on } \{u_\varepsilon=0\},\]
and letting $\varepsilon\to0$ we formally get
\[|\nabla u^+(x_0,t_0)|^2=|A| \partial_tu^+(x_0,t_0).\]
This is the classical Stefan free boundary condition  with a possibly non-constant thermal conductivity coefficient $A$.

\section{Setting and main result}\label{sec main result}
\setcounter{equation}{0}

We are only interested in the interior case. Hence we work in the following setting: suppose $u_{\varepsilon_i}$ is a sequence of solutions to \eqref{eqn} in $Q_1^+:=B_1\times(0,1)\subset\R^n\times\R$ with $\varepsilon_i\to0$, and there exists a constant $\Lambda<+\infty$ such that
  \begin{equation}\label{L infty bound}
 \|u_{\varepsilon_i}\|_{L^\infty(Q_1^+)}\leq\Lambda.
  \end{equation}
Since $L^\infty(Q_1^+)$ is the dual space of $L^1(Q_1^+)$, 
 after passing to a subsequence, we may assume $u_{\varepsilon_i}$ converges to a limit $u$, and $f(u_{\varepsilon_i})$ converges to $\bar{f}$, both $\ast$-weakly in $L^\infty(Q_1^+)$. Similarly, we will also assume $u_{\varepsilon_i}(0)$ converges $\ast$-weakly to a limit $u_0$ in $L^\infty(B_1)$.

The main result of this paper is
\begin{thm}\label{main result}
Under the above assumptions, we have
  \begin{itemize}
    \item [(i)] $u_{\varepsilon_i}^+$ converges to  $u^+$  in  $C_{loc}(Q_1^+)$, and $u_{\varepsilon_i}^-$ converges to $u^-$ $\ast$-weakly in $L^\infty(Q_1^+)$;
    \item [(ii)]$\Omega(t):=\{u(t)>0\}$ is open and increasing in $t$ in the sense that
      \[ \Omega(t_1)\subset\Omega(t_2), \quad \forall \ 0<t_1<t_2<1;\]
      in particular, there exists an upper semi-continuous function $T: B_1\mapsto [0,1]$ such that
\[\Omega=\left\{(x,t)\in Q_1^+: t>T(x)\right\};\]
      \item [(iii)] for a.e.  $0<s<t<1$,
\begin{equation}\label{limit eqn in negative}
u(x,t)=u(x,s)+\int_{s}^{t}\bar{f}(x,\tau)d\tau, \quad \mbox{for a.e. } \ \ x\in  B_1 \setminus \Omega(t).
\end{equation}
        \item [(iv)] in $Q_1^+$, $u^+$ is the weak solution of the generalized, nonlinear Stefan problem associated to the nonlinearity $f$ and
        \begin{equation}\label{2.1}
        W(x):=-u_0(x)-\int_{0}^{T(x)}\bar{f}(x,s)ds;
        \end{equation}
  \item  [(v)]  $\nabla u_{\varepsilon_i}^+$ converges to  $\nabla u^+$ strongly in    $L^2_{loc}(Q_1^+)$.
       \end{itemize}
\end{thm}
\begin{rmk}
  \begin{itemize}
    \item The $L^\infty$ bound usually comes from global consideration, e.g. if there is a suitable initial-boundary value condition.
    \item It is possible that $\bar{f}\neq f(u)$,  which is a common phenomena about nonlinear functions of $\ast$-weakly convergent sequences. But we do have the corresponding weak-$\ast$ convergence from $u_{\varepsilon_i}^\pm$ to $u^\pm$.
    \item By \eqref{limit eqn in negative}, we also have
    \[W(x)=-\lim_{t\to T(x), ~~ t<T(x)}u(x,t)\geq 0 \quad \mbox{ a.e. in } ~ B_1.\]
    The last inequality follows from the fact that $u\leq 0$ in $\{t<T(x)\}$.
    \item When coupled with a suitable initial-boundary value condition, the solution to the limiting problem is unique and we do not need to pass to a subsequence of $\varepsilon\to0$ to choose $\ast$-weakly convergent sequences. The uniqueness  can be proved as in the standard one phase Stefan problem, where one only needs to note that because there is a coupling between $u^+$ and $u^-$ through the free boundary condition,  we need to first prove uniqueness of solutions to the ordinary differential equation satisfied by $u^-$. With this uniqueness result, it is also possible to determine directly   the limiting equation, without resorting to the analysis in Section \ref{sec uniform continuity} and Section \ref{sec completion of proof} of this paper. However, we expect the local interior analysis here could be useful in the study of entire solutions such as travelling waves for  \eqref{competing system} and the nonlinear Stefan problem, which have attracted the attention of many researchers (see e.g. \cite{Du}).
     \end{itemize}
\end{rmk}

The arguments to prove this theorem are mostly standard ones used in the one phase Stefan problem, although we also need to determine the limiting problem for the negative phase. We will consider two transformations of $u_\varepsilon$.
First let $\beta_\varepsilon$ be the inverse of $\alpha_\varepsilon$, that is,
\[
\beta_\varepsilon(v)=
\begin{cases}
  v, & \mbox{if } v\geq 0 \\
  v/\varepsilon, &  \mbox{if } v\leq 0 .
\end{cases}
\]
Then $v_\varepsilon:=\alpha_\varepsilon(u_\varepsilon)$ satisfies
\begin{equation}\label{eqn v}
  \partial_t\beta_\varepsilon(v_\varepsilon)=\Delta v_\varepsilon+f\left(\beta_\varepsilon(v_\varepsilon)\right).
\end{equation}
This is nothing else but  a standard approximation to the enthalpy formulation of one phase Stefan problem. However, in previous work about the one phase Stefan problem this approximation was mainly used to prove the existence of weak solutions, but here we are interested  in the approximation process itself, which, as explained in Section \ref{sec introduction}, is related to the problem \eqref{eqn}.

In \eqref{eqn v}, the information  of $u_\varepsilon^-$ is thrown away, and we can expect good a priori estimates for $v_\varepsilon$, see Section \ref{sec uniform Sobolev} for estimates in Sobolev spaces (which are similar to the one in Friedman \cite{Friedman}), and Section \ref{sec uniform continuity} for estimates about continuity (which is proved as in Caffarelli-Friedman \cite{Caffarelli-Friedman},  Caffarelli-Evans \cite{Caffarelli-Evans} or  DiBenedetto \cite{DiBenedetto}).

Next, we will also use a  variational inequality formulation (see Duvaut \cite{Duvaut}, Friedman and Kinderlehrer \cite{Friedman-K} and Rodrigues \cite{Rodrigues}), where we consider
\[w_\varepsilon(x,t):=\int_{0}^{t}v_\varepsilon(x,s)ds.\]
For $\varepsilon>0$, there is no variational inequality formulation for \eqref{eqn}, but  we will still adopt this name, because the limiting equation of $w_\varepsilon$ is similar to the variational inequality formulation for the one phase Stefan problem.

Important to us is the fact that the equation for $w_\varepsilon$ encodes some information about $u_\varepsilon^-$. This is because it involves integration from $0$ to $t$, which contains a part where $u_\varepsilon<0$. By this observation and noting the increasing property of $\Omega(t)$, the limiting equation for $u_\varepsilon^-$ is determined totally by the one for $w_\varepsilon$.

Finally, in any open set $U\subset Q_1^+\setminus \Omega$, \eqref{limit eqn in negative} says
\begin{equation}\label{1.1}
\partial_tu=\bar{f}  \quad \mbox{in the distributional sense}.
\end{equation}
It will be seen that we do not use any estimate on
\begin{equation}
\partial_tu_\varepsilon -\varepsilon\Delta u_\varepsilon =f(u_\varepsilon ) \quad \mbox{in }~~ U.
\end{equation}
The equation \eqref{limit eqn in negative} is derived solely from the one for $w_\varepsilon$. Note that \eqref{limit eqn in negative} could be stronger than \eqref{1.1} because we do not know if $\partial\Omega$ has zero Lebesgue measure. The defect of this approach is our failure of the determination of the form of $\bar{f}$. We believe this is not achievable in general,  unless  the convergence from $u_\varepsilon(0)$ to $u_0$ is better than the weak-$\ast$ convergence in $L^\infty(B_1)$.

{\bf Notations:} The following notations will be employed in this paper.
\begin{itemize}
   \item We will omit the subscript $i$ and just write $\varepsilon\to0$ for notational simplicity.
   \item An open ball is denoted by $B_r(x)$. If $x=0$, we simply write it as $B_r$.
  \item The  forward parabolic cylinder $Q_r^+(x,t):=B_r(x)\times(t,t+r^2)$, the backward parabolic cylinder $Q_r^-(x,t):=B_r(x)\times (t-r^2,t)$. Finally, the parabolic cylinder is $Q_r(x,t):=B_r(x)\times (t-r^2,t+r^2)$. If the center is $(0,0)$, it will be not written down explicitly.
  \item The notion of almost every property is always understood with respect to the standard Lebesgue measure.
  \item The space $V_2(Q_1^+)$ consists of functions satisfying
  \[\|u\|_{V_2(Q_1^+)}:=\sup_{t\in(0,1)}\|u(t)\|_{L^2(B_1)}+\int_{Q_1^+}|\nabla u(x,t)|^2dxdt<+\infty.\]
    \item We use $L$ to denote the Lipschitz constant of the nonlinearity $f$.
  \item We use $C$ to denote a constant which does not depend on $\varepsilon$. If we want to emphasize its dependence on some quantities, it is written as $C(\cdot)$.
\end{itemize}

The remaining part of the paper is organized as follows. In Section \ref{sec uniform Sobolev} we derive some basic uniform regularity of $u_\varepsilon$ and $v_\varepsilon$ in Sobolev spaces. In Section \ref{sec variational inequality} we introduce the parabolic variational inequality formulation and study the convergence in this framework. Section \ref{sec uniform continuity} is devoted to the proof of uniform convergence of $u_\varepsilon^+$ (equivalently, $v_\varepsilon$). Finally we prove Theorem \ref{main result} in Section \ref{sec completion of proof}.

\section{Uniform Sobolev regularity}\label{sec uniform Sobolev}
\setcounter{equation}{0}

In this section we show   uniform boundedness of $v_\varepsilon$ in some Sobolev spaces.

For any $\eta\in C_0^\infty(B_1)$, multiplying \eqref{eqn} by $u_\varepsilon\eta^2$ and integrating by parts we obtain
\begin{eqnarray}\label{L2 evolution}
  \frac{d}{dt}\frac{1}{2}\int_{B_1}u_\varepsilon^{2}\eta^2&=&-\int_{B_1}|\nabla (u_\varepsilon^+\eta)|^2+\int_{B_1}|u_\varepsilon^+|^2|\nabla\eta|^2\\
  &-&\varepsilon\int_{B_1}|\nabla (u_\varepsilon^-\eta)|^2 +\varepsilon\int_{B_1}|u_\varepsilon^-|^2|\nabla\eta|^2+\int_{B_1}f(u_\varepsilon)u_\varepsilon\eta^2. \nonumber
\end{eqnarray}
Substituting standard cut-off function into this equation we see that  for any $r\in(0,1)$, there exists a constant $C(r)$ such that
\begin{equation}\label{Sobolev on u}
\int_0^1\int_{B_r} \left[|\nabla u_\varepsilon^+|^2+|\partial_tu_\varepsilon^+|^2+\varepsilon\left(|\nabla u_\varepsilon^-|^2+|\partial_tu_\varepsilon^-|^2\right)\right] \leq C(r)\Lambda^2.
\end{equation}
Using $v_\varepsilon$ this is rewritten as
\begin{equation}\label{Sobolev on v}
\left\{ \begin{aligned}
 & \int_0^1\int_{B_r} \left( |\nabla v_\varepsilon^+|^2+|\partial_tv_\varepsilon^+|^2\right)\leq C(r)\Lambda^2,\\
 & \int_0^1\int_{B_r} \left( |\nabla v_\varepsilon^-|^2+|\partial_tv_\varepsilon^-|^2\right)\leq C(r)\Lambda^2\varepsilon.
                          \end{aligned} \right.
\end{equation}

Combining these estimates with the $L^\infty$ bound on $v_\varepsilon$, after passing to a subsequence, we may assume $v_\varepsilon\to u_1$ weakly in $V_2(Q_r^+)$ (for any $r\in(0,1)$) and strongly in $L^2_{loc}(Q_1^+)$. Since
\[\|u_\varepsilon^+-v_\varepsilon\|_{L^\infty(Q_1)}\leq \Lambda\varepsilon,\]
$u_\varepsilon^+$ also converges  strongly to $u_1$ in $L^2_{loc}(Q_1^+)$. After passing to a further subsequence, we also assume $u_\varepsilon^+$ and $v_\varepsilon$ converge a.e. to  $u_1$ in $Q_1^+$.

By these convergence and the weak-$\ast$ convergence of $u_\varepsilon$,
$u_\varepsilon^-$ converges $\ast$-weakly to $u_2:=u_1-u$ in $L^\infty_{loc}(Q_1)$.

\begin{lem}\label{lem 3.1}
We have
\[u_1=u^+, \quad  u_2=u^- \quad \mbox{ a.e. in } Q_1^+.\]
\end{lem}
\begin{proof}
  For a.a. $(x,t)\in\{u_1>0\}$,
\[\lim_{\varepsilon\to 0}u_\varepsilon^+(x,t)=u_1(x,t)>0.\]
Hence for all $\varepsilon$ sufficiently small (depending on the point $(x,t)$), $u_\varepsilon(x,t)=u_\varepsilon^+(x,t)$, and it
converges to $u_1(x,t)$ as $\varepsilon\to0$. In other words, $u_\varepsilon\to u_1$ a.e. in $\{u_1>0\}$.

By  Lebesgue's dominated convergence theorem, for any $\varphi\in L^\infty(Q_1^+)$,
\[\lim_{\varepsilon\to0}\int_{\{u_1>0\}}u_\varepsilon\varphi=\int_{Q_1^+} u_1\varphi,\]
while by the definition of weak-$\ast$ convergence, we  also have
\[\lim_{\varepsilon\to0}\int_{Q_1^+}u_\varepsilon\chi_{\{u_1>0\}}\varphi=\int_{Q_1^+}u\chi_{\{u_1>0\}}\varphi.\]
Therefore $u=u_1$ and $u_2=u_1-u=0$ a.e. in $\{u_1>0\}$. Then $u_1 u_2=0$ a.e. in $Q_1^+$.  The conclusion follows by noting that $u=u_1-u_2$ and both $u_1$ and $u_2$ are nonnegative functions.
\end{proof}


\section{Parabolic variational inequality}\label{sec variational inequality}
\setcounter{equation}{0}
Take an arbitrary $h\in[0,1)$. If $h>0$, after passing to a subsequence, assume $u_\varepsilon(h)$ converges to a limit $u_h$ $\ast$-weakly in $L^\infty(B_1)$.

For $(x,t)\in B_1\times[h,1]$, define
\[w_{h,\varepsilon}(x,t):=\int_{h}^{t}v_\varepsilon(x,s)ds.\]
Direct calculation using \eqref{eqn v} gives
\begin{equation}\label{eqn w}
\Delta w_{h,\varepsilon}-\partial_t w_{h,\varepsilon}=g_{h,\varepsilon},
\end{equation}
where
\begin{eqnarray}\label{form of g}
  g_{h,\varepsilon}(x,t)&:=& \beta_\varepsilon(v_\varepsilon(x,t))-v_\varepsilon(x,t)-\beta_\varepsilon(v_\varepsilon(x,h))
-\int_{h}^{t}f\left(\beta_\varepsilon(v_\varepsilon)(x,s)\right)ds   \nonumber \\
 &=&  -\left(1-\varepsilon\right)u_\varepsilon(x,t)^--u_\varepsilon(x,h)-\int_{h}^{t}f\left(u_\varepsilon(x,s)\right)ds.
\end{eqnarray}

Because $u_\varepsilon$ and $v_\varepsilon$ are uniformly bounded in $L^\infty(Q_1)$, there exists a constant $C$ such that
\[\|g_{h,\varepsilon}\|_{L^\infty(B_1\times(h,1))}\leq C.\]
By definition,
\[ -\Lambda\varepsilon\leq w_{h,\varepsilon} \leq \Lambda, \quad \mbox{in } B_1\times(h,1).\]
Then standard $W^{2,p}$ estimates imply that both $\nabla^2 w_{h,\varepsilon}$ and $\partial_tw_{h,\varepsilon}$
are uniformly bounded in $L^p_{loc}(B_1\times(h,1))$, for any $p<+\infty$. By Sobolev embedding theorems, $w_{h,\varepsilon}$ are uniformly bounded in $C^{1+\alpha,\frac{1+\alpha}{2}}_{loc}(B_1\times[h,1))$  for any $\alpha\in(0,1)$.

 After passing to a subsequence of $\varepsilon\to0$, we may assume
\begin{itemize}
  \item  $w_{h,\varepsilon}$ converges to $w_h$ in $C^{1+\alpha,\frac{1+\alpha}{2}}_{loc}(B_1\times[h,1))$  for any $\alpha\in(0,1)$;
  \item  $\nabla^2w_{h,\varepsilon}$ and $\partial_t w_{h,\varepsilon}$ converge to $\nabla^2w_h$ and $\partial_t w_h$ respectively, with respect to the weak topology in $L^p_{loc}(B_1\times(h,1))$ for any $p<+\infty$;
  \item $g_{h,\varepsilon}$ converges $\ast$-weakly in $L^\infty(B_1\times(h,1))$ to
  \begin{equation}\label{representation of g}
   g_h(x,t):=-u(x,t)^--u_h(x)-\int_{h}^{t}\bar{f}(x,s)ds.
  \end{equation}
\end{itemize}
By the convergence of $v_\varepsilon$ established in Section \ref{sec uniform Sobolev}, $\partial_tw_h=u^+$ in the distributional sense.
Later we will show that all of these limits are independent of the choice of subsequences of $\varepsilon\to0$, and $g_h$ is independent of the choice of $h$.

Passing to the limit in \eqref{eqn w} we get
\begin{equation}\label{limit eqn w}
  \Delta w_h-\partial_tw_h=g_h, \quad \mbox{in } B_1\times(h,1).
\end{equation}
Concerning $w_h$, we observe the following facts:
\begin{enumerate}
  \item since $\partial_tw_h\geq 0$, $\{w_h(t)>0\}$ (which is an open set by the continuity of $w_h$) is increasing in $t$ in the sense that
      \[\{w_h(t_1)>0\}\subset \{w_h(t_2)>0\}, \quad \mbox{for any} ~~ h<t_1<t_2<1;\]
  \item $w_h\geq 0$, which follows by combining (1) with the fact that $w_h(x,h)\equiv 0$;
  \item $g_h=0$ and $u^+=0$ a.e. in $\{w_h=0\}$, which is a consequence of the fact that $w_h\in W^{2,1}_{p,loc}(B_1\times(h,1))$.
\end{enumerate}
\begin{rmk}\label{rmk 4.1}
   Unlike the standard approximation to the parabolic variational inequality formulation of the one phase Stefan problem (such as the one used in \cite[Section 2]{Friedman-K}), where basically one considers an approximate equation 
   \[\Delta\widetilde{w}_\varepsilon-\partial_t\widetilde{w}_\varepsilon+\beta_\varepsilon\left(\widetilde{w}_\varepsilon-\varepsilon\right)=0,\]
       at this stage we cannot use  Eqns. \eqref{eqn w}-\eqref{limit eqn w} to prove that $u>0$ a.e. in $\{w_h>0\}$, by noting that  after simplification, \eqref{eqn w} reads as
  \[\Delta w_{h,\varepsilon}-\beta_\varepsilon(\partial_tw_{h,\varepsilon})=u_\varepsilon(x,h).\]

In fact,   by \eqref{representation of g},  even if $f=0$ (hence $\bar{f}=0$), it is not clear if $\partial_tg_h=0$ in $\{w_h>0\}$ (in the distributional sense). Although this is indeed the case, we need to first prove that $u^-=0$ in $\{w_h>0\}$  (see Section \ref{sec completion of proof}).
\end{rmk}

\section{Uniform convergence}\label{sec uniform continuity}
\setcounter{equation}{0}

In this section we prove the uniform convergence of $u_\varepsilon^+$ to $u^+$. This is in fact a direct consequence of results in
\cite{Caffarelli-Friedman, Caffarelli-Evans, DiBenedetto}. In particular, the next four lemmas, Lemma \ref{lem De Giorgi iteration}-Lemma \ref{lem Harnack}, are just suitable adaption of corresponding results in \cite{Caffarelli-Friedman, Caffarelli-Evans, DiBenedetto} to our specific setting. These four lemmas are used to prove the continuity of $v_\varepsilon$ in backward parabolic cylinders. Continuity in forward parabolic cylinders can also be proved following the methods in \cite{Caffarelli-Friedman, Caffarelli-Evans, DiBenedetto}. However, here we present a direct proof (see Proposition \ref{prop uniform continuity}), by establishing an almost-monotonicity in time property for $u_\varepsilon$ (see Lemma \ref{lem uniform continuity from above}). This lemma will also be used in the next section to prove $\{u>0\}=\{w_h>0\}$.

In the following we denote, for any $r>0$,
\[Q_r^\ast:=B_r\times \left(-\frac{r^2}{4n},0\right).\]

\begin{lem}\label{lem De Giorgi iteration}
For any $M>0$, there exists a constant  $\sigma(M)>0$ so that the following holds. Suppose $v_\varepsilon$ is a continuous solution of \eqref{eqn v} (with a possibly different nonlinearity $f$, but its Lipschitz constant is still bounded by $L$) in $Q_1^-$ satisfying $-M \varepsilon\leq v_\varepsilon\leq 1$ and
\begin{equation}\label{L2 small}
 \frac{1}{|Q_1^\ast|}\int_{Q_1^\ast}\left(1-v_\varepsilon\right)\leq \sigma.
\end{equation}
Then
\[ v_\varepsilon\geq \frac{1}{2} \quad \mbox{in} \quad Q_{1/2}^\ast.\]
\end{lem}
\begin{proof}
The proof is divided into two steps.

{\bf Step 1. Caccioppoli inequality.}
Given $-1/(4n)<t_1<t_2<0$ and $k\in(0,1)$, for any $\eta\in C_0^\infty(B_1\times (t_1,t_2))$ and $t\in(t_1,t_2)$, multiplying \eqref{eqn v} by $(k-v_\varepsilon)_+\eta^2$ and integrating in $B_1\times (t_1,t)$ leads to
\begin{eqnarray}\label{Caccioppoli 1}
   && \int_{t_1}^{t}\int_{B_1}|\nabla \left[(k-v_\varepsilon)_+\eta\right]|^2+\int_{B_1}\mathcal{B}_{k,\varepsilon}(v_\varepsilon(t))\eta^2 \\
  &=& \int_{t_1}^{t}\int_{B_1}\left[(k-v_\varepsilon)_+^2|\nabla\eta|^2+2\mathcal{B}_{k,\varepsilon}(v_\varepsilon(t))\eta \partial_t\eta- f(u_\varepsilon)(k-v_\varepsilon)_+\eta^2\right], \nonumber
\end{eqnarray}
where
\[
\mathcal{B}_{k,\varepsilon}(v):=
\begin{cases}
  0, & \mbox{if } v\geq k \\
  \frac{1}{2}(k-v)^2, & \mbox{if } 0\leq v \leq k \\
  \frac{k^2}{2}+\frac{1}{2\varepsilon}v^2-\frac{k}{\varepsilon}v, & \mbox{otherwise}.
\end{cases}
\]

Since $v_\varepsilon\geq -M \varepsilon$, there exists a constant $C(M)$ depending only on $M$ such that
\[\frac{1}{2}(k-v_\varepsilon)_+^2 \leq \mathcal{B}_{k,\varepsilon}(v_\varepsilon) \leq C(M)(k-v_\varepsilon)_+^2.\]
With this estimate \eqref{Caccioppoli 1} is transformed into
\begin{eqnarray}\label{Caccioppoli}
   && \int_{t_1}^{t}\int_{B_1}|\nabla \left[(k-v_\varepsilon)_+\eta\right]|^2+\int_{B_1}(k-v_\varepsilon(t))_+^2 \eta^2 \\
  &\leq &C(M) \int_{t_1}^{t}\int_{B_1}\left[(k-v_\varepsilon)_+^2\left(|\nabla\eta|^2+|\eta\partial_t\eta|\right)+(k-v_\varepsilon)_+\eta^2\right]. \nonumber
\end{eqnarray}

Since $t$ is arbitrary, by this inequality and Sobolev embedding theorem we get two constants $p>1$ and $C(M)$ such that
\begin{equation}\label{Sobolev}
\left(\int_{t_1}^{t_2}\int_{B_1} (k-v_\varepsilon)_+^{2p}|\eta|^{2p}\right)^{\frac{1}{p}}\leq C(M) \int_{t_1}^{t_2}\int_{B_1}(k-v_\varepsilon)_+\left[|\nabla\eta|^2+|\eta\partial_t\eta| +\eta^2\right].
\end{equation}

{\bf Step 2. De Giorgi iteration.} For any $m\geq 1$, set
\[k_m:=2^{-1}+2^{-m}, \quad r_m=2^{-1}+2^{-m}\]
and
\[a_m:=\int_{Q_{r_m}^\ast}(k_m-v_\varepsilon)_+.\]

Take a function $\eta_m\in C_0^\infty(Q_{r_m}^\ast)$ such that $\eta_m\equiv 1$ in $Q_{r_{m+1}}^\ast$, $0\leq \eta_m\leq 1$ and $|\nabla\eta_m|^2+|\partial_t\eta_m|\leq 16(r_m-r_{m+1})^{-2}$. Substituting $\eta_m$ into \eqref{Sobolev} leads to
\begin{equation}\label{iteration 1}
 \left(\int_{Q_{r_{m+1}}^\ast} (k_m-v_\varepsilon)_+^{2p}\right)^{\frac{1}{p}}\leq C(M)4^ma_m.
\end{equation}
In $\{(k_{m+1}-v_\varepsilon)_+\neq 0\}$, we have
\[(k_m-v_\varepsilon)_+\geq 2^{-m-1}.\]
Therefore
\[\int_{Q_{r_{m+1}}^\ast}(k_m-v_\varepsilon)_+^{2p}\geq 2^{-(m+1)(2p-1)} \int_{Q_{r_{m+1}}^\ast}(k_{m+1}-v_\varepsilon)_+.\]
Substituting this into \eqref{iteration 1} we get a constant $A(M)>1$ such that
\begin{equation}\label{iteration 2}
a_{m+1}\leq  A(M)^m a_m^p.
\end{equation}

By our assumption on $v_\varepsilon$, we have
\[a_1=\int_{Q_1^\ast}(1-v_\varepsilon)_+ \leq  \int_{Q_1^\ast}(1-v_\varepsilon)\leq  \sigma |Q_1^\ast|.\]
By \cite{Ladyzenskaja}, if $\sigma$ is small enough (depending only on $A(M)$), then $\lim_{m\to+\infty}a_m=0$. Hence
\[\int_{Q_{1/2}^\ast}\left(\frac{1}{2}-v_\varepsilon\right)_+=0.\]
Since $v_\varepsilon$ is continuous, this implies that $v_\varepsilon\geq 1/2$ in $Q_{1/2}^\ast$.
\end{proof}

\begin{lem}\label{lem decay}
For any  $\varepsilon,\sigma,M>0$, there exist two constants $\theta:=\theta(\varepsilon,\sigma,M)\in(0,1)$ and $\rho:=\rho(\varepsilon,\sigma,M)\in(0,1)$  so that the following holds. Suppose $v_\varepsilon$ is a continuous solution of \eqref{eqn v} in $Q_1^\ast$ (with a possibly different nonlinearity $f$, but its Lipschitz constant is still bounded by $L$)  satisfying $v_\varepsilon\geq -M \varepsilon$,  $\sup_{Q_1^\ast}v_\varepsilon=1$ and
\[\frac{1}{|Q_1^\ast|}\int_{Q_1^\ast}v_\varepsilon\leq 1-\sigma,\]
 then
\[ v_\varepsilon\leq \theta \quad \mbox{in} \quad  Q_{\rho}^\ast.\]
\end{lem}
\begin{proof}
Take
\[\gamma:=1-\frac{\sigma}{2}, \quad \beta:=1-\frac{1-\sigma}{1-\sigma/2}.\]
Then it is readily verified that
\[|\{v_\varepsilon<\gamma\}\cap Q_1^\ast|\geq \beta|Q_1^\ast|.\]
Since
\[\partial_t v_\varepsilon^+-\Delta v_\varepsilon^+\leq Lv_\varepsilon^+,\]
the conclusion follows by applying \cite[Lemma 3.1]{Caffarelli-Friedman}.
\end{proof}

\begin{lem}\label{lem uniform continuity from above}
For any $\delta>0$, there exists an $r(\delta)<\delta$ (independent of $\varepsilon$) such that, for any $(x,t)\in\{v_\varepsilon\leq 0\}\cap Q_{1-\delta}^-(0,1)$,
\[\sup_{Q_{r(\delta)}^\ast(x,t)}v_\varepsilon\leq\delta.\]
\end{lem}
\begin{proof}
  Assume by the contrary, there exists a $\delta>0$, a sequence of points $(x_\varepsilon,t_\varepsilon)\in Q_{1-\delta}^-(0,1)$ satisfying $v_\varepsilon(x_\varepsilon,t_\varepsilon)\leq0$, and a sequence of $r_\varepsilon\to0$ such that
  \begin{equation}\label{absurd assumption}
    \sup_{Q_{r_\varepsilon}^\ast(x_\varepsilon,t_\varepsilon)}v_\varepsilon\geq \delta.
  \end{equation}

For any $r>0$, denote
\[M_\varepsilon(r):=\sup_{Q_r^\ast(x_\varepsilon,t_\varepsilon)}v_\varepsilon.\]

For any $\varepsilon>0$  and $r\in (r_\varepsilon, \delta)$, take the rescaling
\[v^r_\varepsilon(x,t):=\frac{1}{M_\varepsilon(r)}v_\varepsilon(x_\varepsilon+r x, t_\varepsilon+r^2 t).\]
It is a continuous solution of
\[\partial_t \beta_\varepsilon(v_\varepsilon^r)=\Delta v_\varepsilon^r-\frac{r^2}{M_\varepsilon(r)}f\left(M_\varepsilon(r) \beta_\varepsilon(v_\varepsilon^r)\right).\]
The Lipschitz constant of the nonlinearity is still bounded by $L$. Furthermore,  $v_\varepsilon^r\geq -M\varepsilon$ in $Q_1^\ast$ and $\sup_{Q_1^\ast}v_\varepsilon^r=1$, where $M:=\Lambda/\delta$.  Since $v^r_\varepsilon(0,0)\leq 0$, by Lemma \ref{lem De Giorgi iteration} we must have
\[\frac{1}{|Q_1^\ast|}\int_{Q_1^\ast}\left(1-v^r_\varepsilon\right)\geq \sigma(M).\]
Therefore Lemma \ref{lem decay} is applicable, which gives
\[\sup_{Q_{\rho}^\ast}v_\varepsilon^r\leq \theta:=\theta(\sigma(M),M)<1.\]
Rescaling back this is
\[ M_\varepsilon\left(\rho r\right)\leq \theta M_\varepsilon(r),  \quad \forall r\in (2r_\varepsilon, \delta/2).\]
An iteration of this estimate leads to
\[ M_\varepsilon(r_\varepsilon)\leq \theta^{\log (\frac{\delta}{8r_\varepsilon})/|\log \rho|}M_\varepsilon\left(\rho\delta\right).\]
Since $\lim_{\varepsilon\to0}r_\varepsilon=0$ and $M_\varepsilon(\rho\delta)\leq \Lambda$, we obtain
\[\lim_{\varepsilon\to0}M_\varepsilon(r_\varepsilon)=0,\]
which is a contradiction with \eqref{absurd assumption}.
\end{proof}

\begin{lem}\label{lem Harnack}
  For any $\lambda>0$, there exists a $\rho(\lambda)>0$ (independent of $\varepsilon$) so that the following holds. If $(x_\varepsilon,t_\varepsilon)\in Q_1^+$ and $v_\varepsilon(x_\varepsilon,t_\varepsilon)\geq\lambda$, then
  \[ v_\varepsilon\geq \frac{\lambda}{2} \quad \mbox{in } Q_{\rho(\lambda)}^\ast(x_\varepsilon,t_\varepsilon).\]
\end{lem}
\begin{proof}
  Define $M_\varepsilon(r)$ and $v_\varepsilon^r$ as in the proof of the previous lemma. Note that $M_\varepsilon(r)\geq\lambda$ for any $r$. We still have $v_\varepsilon^r\geq -M\varepsilon$ in $Q_1^\ast$ and $\sup_{Q_1^\ast}v_\varepsilon^r=1$, with $M:=\Lambda/\lambda$.

 Lemma \ref{lem De Giorgi iteration}
says
\begin{equation}\label{alternative 1}
\frac{1}{|Q_1^\ast|}\int_{Q_1^\ast}\left(1-v_\varepsilon^r\right)\leq \sigma(M) \Longrightarrow \inf_{Q_{\rho r}^\ast(x_\varepsilon,t_\varepsilon)}v_\varepsilon\geq \frac{1}{2}M_\varepsilon(r),
\end{equation}
while Lemma  \ref{lem decay} says
  \begin{equation}\label{alternative 2}
\frac{1}{|Q_1^\ast|}\int_{Q_1^\ast}\left(1-v_\varepsilon^r\right)\geq \sigma(M) \Longrightarrow M_\varepsilon(\rho r)\leq \theta M_\varepsilon(r).
\end{equation}

Let $k_0$ be the minimal positive integer such that \eqref{alternative 1} holds for $r=\rho^{k_0}$. Then for any $k<k_0$, by \eqref{alternative 2} we have
\[M_\varepsilon(\rho^{k+1})\leq \theta M_\varepsilon(\rho^k).\]
Since for any $r$, $\lambda\leq M_\varepsilon(r)\leq \Lambda$, we must have
\[k_0\leq \frac{\log\left(\Lambda/\lambda\right)}{|\log\theta|}.\]
Take $\rho(\lambda):=\rho^{\frac{\log\left(\Lambda/\lambda\right)}{|\log\theta|}}$. By the definition of $k_0$, it satisfies the assumption in \eqref{alternative 1}. Because $M_\varepsilon(\rho^{k_0})\geq \lambda$, we get
\[
  \inf_{Q_{\rho(\lambda)}^\ast(x_\varepsilon,t_\varepsilon)}v_\varepsilon \geq  \inf_{Q_{\rho^{k_0}}^\ast(x_\varepsilon,t_\varepsilon)}v_\varepsilon
    \geq \frac{1}{2}M_\varepsilon(\rho^{k_0})\geq\frac{\lambda}{2}. \qedhere
\]
\end{proof}

\begin{rmk}
  With more work it is possible to obtain a modulus of continuity for $v_\varepsilon$, which holds uniformly in $\varepsilon$. We do not need this and a qualitative result is sufficient for our purpose.
\end{rmk}

The following lemma shows that for solutions of \eqref{eqn}, the positivity set is almost increasing in time. This will be passed to the $\varepsilon\to0$ limit as a real monotonicity property.
\begin{lem}\label{lem expanding property}
Given $\lambda>0$ and a cylinder $Q^T:=B_r(x)\times(t,t+T)$, suppose $u_\varepsilon$ is a continuous solution of \eqref{eqn} in this cylinder, satisfying
\begin{itemize}
  \item  $u_\varepsilon\geq -\Lambda $ in $Q^T$;
  \item  $u_\varepsilon\geq \lambda$ in $B_{2r/3}(x)\times\{t\}$.
\end{itemize}
If $\varepsilon$ is sufficiently small, then $u_\varepsilon>0$ in $B_{r/4}(x)\times (t+r^2/8,t+T)$.
\end{lem}
\begin{proof}
Without loss of generality, assume $(x,t)=(0,0)$ and $r=1$. We divide the proof into three steps.

{\bf Step 1. Comparison functions.}
Take a  function $\varphi\in C^\infty(\R^n)$,  satisfying  $\varphi=\lambda$ in $B_{1/2}(0)$, $\varphi\leq\lambda$ everywhere and $\varphi\equiv -\Lambda$ outside $B_{2/3}(0)$. Let $u_{\varepsilon,\ast}$ be the solution of \eqref{eqn} in $Q^T$ with initial-boundary value $\varphi$. By the comparison principle for \eqref{eqn} (see \cite[Theorem 6]{Cannon-Hill}),
\begin{equation}\label{subsolution}
u_\varepsilon\geq u_{\varepsilon,\ast} \quad \mbox{in } Q^T.
\end{equation}
 Therefore we only need to show that $u_{\varepsilon,\ast}>0$ in $B_{1/4}\times(1/8,T)$.

Set $v_{\varepsilon,\ast}:=\alpha_\varepsilon(u_{\varepsilon,\ast})$ and
\[w_{\varepsilon,\ast}(x,t):=\int_{0}^{t}v_{\varepsilon,\ast}(x,s)ds.\]
As in Section \ref{sec uniform Sobolev} and Section \ref{sec variational inequality}, we assume $w_{\varepsilon,\ast}\to w_\ast$ and $v_{\varepsilon,\ast}\to v_\ast$ in the corresponding sense.

Since
\[\partial_tu_{\varepsilon,\ast}^--\varepsilon\Delta u_{\varepsilon,\ast}^-\leq Lu_{\varepsilon,\ast}^-,\]
by the comparison principle
\begin{equation}\label{comparison with linear}
 u_{\varepsilon,\ast}^- \leq e^{Lt}\widetilde{u}_\varepsilon \quad \mbox{in } Q^T.
\end{equation}
Here $\widetilde{u}_\varepsilon$ is the solution of
\begin{equation}\label{linear vanishing viscosity}
  \left\{ \begin{aligned}
 &\partial_t\widetilde{u}_\varepsilon=\varepsilon\Delta  \widetilde{u}_\varepsilon,  \quad &\text{in}~~Q^T,\\
 &\widetilde{u}_\varepsilon=\varphi^-,  \quad &\text{on} ~~\partial^pQ^T.
                          \end{aligned} \right.
\end{equation}

{\bf Step 2. A positive lower bound.}  Let $G(x,y,t)$ be the heat kernel for the standard heat operator $\partial_t-\Delta$ in $B_1$ with zero Dirichlet boundary condition. Then the heat kernel for $\partial_t-\varepsilon\Delta$ is $G(x,y,\varepsilon t)$. Hence we have the representation formula
\begin{equation}\label{representation for heat}
  \widetilde{u}_\varepsilon(x,t)=\Lambda-\int_{B_1}G(x,y,\varepsilon t)\left[\Lambda-\varphi^-(y)\right]dy.
\end{equation}
From this representation and the continuity of $\varphi^-$, we get
\begin{equation}\label{vanishing viscosity limit for linear heat}
  \lim_{\varepsilon\to0}\widetilde{u}_\varepsilon(x,t)=\varphi^-(x), \quad \mbox{uniformly in } B_{2/3}\times[0,T].
\end{equation}

Combining \eqref{comparison with linear} and \eqref{vanishing viscosity limit for linear heat}, we see as $\varepsilon\to0$, $u_{\varepsilon,\ast}^-\to 0$ uniformly in $B_{1/2}\times [0,T]$. 
Substituting this into \eqref{form of g} gives
\[\Delta w_{\varepsilon,\ast}(x,t)-\partial_tw_{\varepsilon,\ast}(x,t)=\varphi(x,0)-\int_{0}^{t}f(u_{\varepsilon,\ast}^+)+o(1), \quad \mbox{in }~~B_{1/2}\times (0,T),\]
where $o(1)$ is measured in $L^\infty$.
Passing to the limit  we obtain
\[\Delta w_\ast(x,t)-\partial_t w_\ast(x,t)=\varphi(x,0)-\int_{0}^{t}f(v_\ast), \quad \mbox{ in }~~B_{1/2}\times (0,T).\]
Taking derivative in $t$ and noting that $v_\ast=\partial_t w_\ast$, we get
\[\partial_t v_\ast-\Delta v_\ast =f(v_\ast)  \quad \mbox{  in }~~B_{1/2}\times (0,T).\]
The above three equations are all understood in the distributional sense. However, by standard parabolic theory, $v_\ast$ is smooth in $B_{1/2}\times (0,T)$.
As before we still have $v_\ast\geq 0$.  By  Harnack inequality,  there exists a constant $\gamma>0$ such that
\begin{equation}\label{strictly positive}
  v_\ast>\gamma \quad \mbox{strictly in }~~B_{1/3}\times [1/16,T].
\end{equation}

{\bf Step 3. Completion of the proof.}  By Section \ref{sec uniform Sobolev}, as $\varepsilon\to0$, $v_{\varepsilon,\ast}$ converges strongly to $v_\ast$ in $L^1_{loc}(Q_1^+)$. Thus for any $(x,t)\in B_{1/4}\times(1/8,T)$,
\[\lim_{\varepsilon\to 0}\int_{Q_{1/16}^{\ast}(x,t)}v_{\varepsilon,\ast}=\int_{Q_{1/16}^{\ast}(x,t)}v_\ast\geq \gamma|Q_{1/16}^\ast|.\]
Combining this estimate with Lemma \ref{lem uniform continuity from above}, we get  an $\varepsilon_\ast$ such that
\[ v_{\varepsilon,\ast}>0 \quad \mbox{in } B_{1/4}\times[1/8,T], \quad \mbox{if } \varepsilon\leq\varepsilon_\ast.\]
The proof is complete by noting that $v_\varepsilon\geq v_{\varepsilon,\ast}$ in $Q^T$.
\end{proof}

\begin{prop}\label{prop uniform continuity}
  As $\varepsilon\to0$, $v_\varepsilon\to u^+$ in $C_{loc}(Q_1^+)$.
\end{prop}
\begin{proof}
For any $(x,t)\in Q_1^+$, let
\begin{equation}\label{limsup}
 \left\{ \begin{aligned}
   & v^\ast(x,t):=\limsup_{\varepsilon\to0 \ \ \mbox{and} \ \  (x_\varepsilon,t_\varepsilon)\to(x,t)}v_\varepsilon(x_\varepsilon,t_\varepsilon), \\
   & v_\ast(x,t):=\liminf_{\varepsilon\to0 \ \ \mbox{and} \ \  (x_\varepsilon,t_\varepsilon)\to(x,t)}v_\varepsilon(x_\varepsilon,t_\varepsilon).
   \end{aligned} \right.
\end{equation}
To prove the uniform convergence of $v_\varepsilon$ to $u^+$, in view of the a.e. convergence of $v_\varepsilon$, it is sufficient to show that $v_\ast=v^\ast$ everywhere. Since  we always have $0\leq v_\ast\leq v^\ast\leq \Lambda$, this is trivially true if $v^\ast(x,t)=0$.

It remains to consider the case when $\lambda:=v^\ast(x,t)>0$. Take a subsequence $\varepsilon_i\to0$ and a sequence of points $(x_{\varepsilon_i},t_{\varepsilon_i})\to(x,t)$ to attain the limsup in \eqref{limsup}.
 By Lemma \ref{lem Harnack},  there exists a $\rho:=\rho(\lambda/2)$  such that  $v_{\varepsilon_i}\geq \lambda/4$ in $Q_{\rho}^-(x_{\varepsilon_i},t_{\varepsilon_i})$. Then Lemma \ref{lem expanding property}, applied to $v_{\varepsilon_i}$ in the cylinder $B_\rho(x_{\varepsilon_i})\times [t_{\varepsilon_i}-\rho^2,t_{\varepsilon_i}+\rho^2]$, implies that $v_{\varepsilon_i}>0$ strictly in $Q_{\rho/4}^+(x_{\varepsilon_i},t_{\varepsilon_i})$. Hence for all $\varepsilon_i$ small,
\[\partial_tv_{\varepsilon_i}-\Delta v_{\varepsilon_i}=f(v_{\varepsilon_i}), \quad  \mbox{in } Q_{\rho/6}(x,t).\]
Since $0<v_{\varepsilon_i}<\Lambda$ in $Q_{\rho/6}(x,t)$, standard parabolic regularity theory and Arzela-Ascoli theorem imply that $v_{\varepsilon_i}$
converge to $u^+$ in a smooth way in $Q_{\rho/7}(x,t)$. As a consequence,
\[\partial_tu^+-\Delta u^+=f(u^+), \quad  \mbox{in } Q_{\rho/7}(x,t).\]
By Harnack inequality, there exists a constant $\gamma>0$ such that $u^+\geq \gamma$ in $Q_{\rho/8}(x,t)$.

Since $v_\varepsilon\to u^+$ in $L^1(Q_{\rho/8}(x,t))$, similar to Step 3 in the proof of Lemma \ref{lem expanding property}, we deduce that for all $\varepsilon$ small (before passing to the subsequence), $v_\varepsilon>0$ and converges uniformly to $u^+$ in $Q_{\rho/10}(x,t)$. In particular, $v_\ast(x,t)=v^\ast(x,t)$.
\end{proof}
Because $\|u_\varepsilon^+-v_\varepsilon\|_{L^\infty(Q_1)}\leq \Lambda\varepsilon$, we get
\begin{coro}\label{coro uniform convergence}
   As $\varepsilon\to0$, $u_\varepsilon^+\to u^+$ in $C_{loc}(Q_1^+)$.
\end{coro}

\section{Completion of  the proof}\label{sec completion of proof}
\setcounter{equation}{0}

Since $v_\varepsilon\to u^+$ in $C_{loc}(Q_1^+)$, $u^+\in C(Q_1^+)$. Hence $\Omega:=\{u>0\}$ is an open subset of $Q_1^+$, and for each $t\in(0,1)$, $\Omega(t):=\{u(t)>0\}$ is an open subset of $B_1$.

\smallskip

{\bf Property 1.} $\partial_tu -\Delta u =f(u )$ and $\bar{f}=f(u)$ in $\Omega$.
\begin{proof}
This follows from the local uniform convergence of $u_\varepsilon^+$ to $u$ in $\Omega$.
\end{proof}

{\bf Property 2.} $\Omega(t)$ is increasing in $t$.
\begin{proof}
Take a point $x\in\Omega(t)$. By the continuity of $u^+$, there exists a cylinder $Q_r(x)$ and a positive constant $\lambda>0$ such that $u\geq 2\lambda $ in $Q_{r}(x)$. By the uniform convergence of $u_\varepsilon^+$, for all $\varepsilon$ small, $u_\varepsilon\geq \lambda$ in $Q_{r}(x)$. Applying Lemma \ref{lem expanding property} to $u_\varepsilon$ in the cylinder $B_r(x)\times(t-r^2,1)$, we see for these $\varepsilon$, $u_\varepsilon>0$ in $B_r(x)\times(t-r^2/4,1)$. Hence
\[\partial_tu_\varepsilon-\Delta u_\varepsilon=f(u_\varepsilon), \quad \mbox{in } B_r(x)\times(t-r^2/4,1).\]
Then we find a constant $\gamma>0$ by  Harnack inequality   such that $u_\varepsilon\geq \gamma$ in $B_{r/2}(x)\times[t,1]$.
 Letting $\varepsilon\to0$ and using the uniform convergence of $u_\varepsilon^+$, we  deduce that $B_{r/2}(x)\times[t,1]\subset \Omega$.
\end{proof}

By this monotonicity property, we obtain a function $T: B_1\mapsto [0,1]$ such that
\[\Omega=\left\{(x,t)\in Q_1^+: t>T(x)\right\},\]
which is the waiting time of $x$. By the continuity of $u^+$, $T$ is an upper semi-continuous function.

{\bf Property 3.} In $B_1\times(h,1)$, $\{w_h>0\}=\Omega$.
\begin{proof}
For any $(x,t)\in Q_1^+\setminus\Omega$, by Property 2,
\[u(x,s)\leq 0, \quad \mbox{a.e. in } (0,t).\]
Therefore Lemma \ref{lem 3.1} and Proposition \ref{prop uniform continuity} together imply that
\[ v_\varepsilon(x,s)\to 0, \quad \mbox{uniformly in }\ s\in[0,t].\]
Then by the definition of $w_{h,\varepsilon}$, we get
 \[w_h(x,s)= 0, \quad \forall s\in[h,t].\]
 In conclusion, $w_h=0$ outside $\Omega$.

 Next we claim that $w_h>0$ in $\Omega$. Indeed, for any $(x,t)\in \Omega$, there exists an open neighborhood of it where $u>0$ strictly. The claim then follows from the fact that $\partial_tw_h=u^+$.
\end{proof}

{\bf Property 4.} For any $0\leq h<1$,
\begin{equation}\label{5.4}
u(x,t)=u_h(x)+\int_{h}^{t}\bar{f}(x,s)ds, \quad \mbox{for a.e. } \ \ (x,t)\in \left[B_1\times(h,1)\right]\setminus \Omega.
\end{equation}
\begin{proof}
By Property 3, $\left[B_1\times(h,1)\right]\setminus \Omega\subset \{w_{h}=0\}$. By results in Section \ref{sec variational inequality}, $g_{h}=0$ a.e. in $\left[B_1\times(h,1)\right]\setminus \Omega$. The conclusion then follows from the formula of $g_h$ in \eqref{representation of g}.
\end{proof}

{\bf Property 5.} For a.e. $h\in(0,1)$,
\[
u(x,h)=u_h(x), \quad \mbox{for a.e. } \ \ x\in  B_1 \setminus \Omega(h).
\]
\begin{proof}
Taking $h=0$ in \eqref{5.4} we obtain
\begin{equation}\label{5.5}
u(x,t)=u_0(x)+\int_{0}^{t}\bar{f}(x,s)ds, \quad \mbox{for a.e. } \ \ (x,t)\in \left[B_1\times(0,1)\right]\setminus \Omega.
\end{equation}
Subtracting \eqref{5.4} from \eqref{5.5} we obtain
\[u_h(x)=u_0(x)+\int_{0}^{h}\bar{f}(x,s)ds \quad \mbox{for a.e. } \ \ x\in  B_1 \setminus \Omega(h).\]
Comparing this with the $t=h$ case of \eqref{5.5} we conclude the proof.
\end{proof}
Combining Property 4 with Property 5 we obtain \eqref{limit eqn in negative}.

\medskip

{\bf Property 6.} For any $(x,t)\in\Omega$,
\begin{equation}\label{eqn w limit}
  \Delta w_0(x,t)-\partial_tw_0(x,t)=-u_0(x)-\int_{0}^{T(x)}\bar{f}(x,s)ds-\int_{T(x)}^{t}f(u(x,s))ds.
\end{equation}
\begin{proof}
This follows by dividing the integral in \eqref{representation of g} (with $h=0$) into two parts: $(0,T(x))$ and $(T(x),t)$, and then applying Property 1 to the latter one.
\end{proof}

Differentiating \eqref{eqn w limit} in $t$ gives the equation for $u^+$, \eqref{limiting eqn}, with $W$ given by \eqref{2.1}.

\medskip

{\bf Property 7.} $\nabla u_\varepsilon^+\to \nabla u^+$ strongly in $L^2_{loc}(Q_1^+)$.
\begin{proof}
  Since $u^+$ is continuous and $\partial_tu-\Delta u=f(u)$ in $\{u>0\}$, for any $\eta\in C_0^\infty(Q_1^+)$ and $k>0$,
  we have
  \begin{equation}\label{5.1}
 \int_{Q_1^+} -|(u-k)^+|^2\eta\partial_t\eta+|\nabla(u-k)^+|^2\eta^2+2\eta (u-k)^+\nabla (u-k)^+\cdot\nabla\eta-f(u)(u-k)_+\eta^2=0.
  \end{equation}
  Letting $k\to0$ we get
   \begin{equation}\label{5.2}
 \int_{Q_1^+} -|u^+|^2\eta\partial_t\eta+|\nabla u^+|^2\eta^2+2\eta u^+\nabla u^+\cdot\nabla\eta-f(u^+)u^+\eta^2=0.
  \end{equation}
  On the other hand, testing \eqref{eqn} with $u_\varepsilon^+\eta^2$ we obtain
   \begin{equation}\label{5.3}
 \int_{Q_1^+} -|u_\varepsilon^+|^2\eta\partial_t\eta+|\nabla u_\varepsilon^+|^2\eta^2+2\eta u_\varepsilon^+\nabla u_\varepsilon^+\cdot\nabla\eta-f(u_\varepsilon^+)u_\varepsilon^+\eta^2=0.
  \end{equation}
 Letting $\varepsilon\to0$ in \eqref{5.3}, by the strong convergence of $u_\varepsilon^+$ and weak convergence of $\nabla u_\varepsilon^+$ in $L^2_{loc}(Q_1^+)$, we obtain
 \begin{eqnarray*}
 \lim_{\varepsilon\to0} \int_{Q_1^+} |\nabla u_\varepsilon^+|^2\eta^2&=&\int_{Q_1^+}  |u^+|^2\eta\partial_t\eta-2\eta u^+\nabla u^+\cdot\nabla\eta+f(u^+)u^+\eta^2\\
 &=&\int_{Q_1^+} |\nabla u^+|^2\eta^2.
 \end{eqnarray*}
 This gives the strong convergence of $\nabla u_\varepsilon^+$ in $L^2_{loc}(Q_1^+)$.
\end{proof}

The proof of Theorem \ref{main result} is complete.

\end{document}